\documentclass[12 pt]{amsart}
\usepackage{amscd,amssymb,amsthm}
\usepackage[arrow,matrix]{xy}
\usepackage{graphicx}
\theoremstyle{plain}

\theoremstyle{definition}

\numberwithin{equation}{section} 
\newtheorem{thm}[subsection]{Theorem}
\newtheorem{lem}[section]{Lemma}

\newtheorem{prop}[subsection]{Proposition}

\newtheorem{cor}[subsection]{Corollary}
\newtheorem{notn}[section]{Notation}
\newtheorem{expl}[subsection]{Example}
\newtheorem{exple}[subsubsection]{Example}

\newtheorem{definition}[section]{\textmd{Definition}}

\let\cite\citep
\def\qed{\ifhmode
\unskip\nobreak\hfill\vrule height5pt width4pt depth2pt\medskip\fi
\ifmmode\eqno{\vrule height5pt width4pt depth2pt}\fi}
\def\qed{\ifhmode\unskip\nobreak\fi\quad\ifmmode\Box\else$\Box$\fi}

\begin{document}
\author{Zahid Raza, Imran and  Bijan Davvaz}

\title{On the Structure of Involutions and  Symmetric Spaces of Quasi Dihedral Group}
 \address{Department of Mathematics \\ College of Sciences,University of Sharjah, UAE.}
\email{zraza@sharjah.ac.ae}

\address{Department of Mathematics \\ National University of Computer and Emerging Sciences, Lahore, Pakistan.}
\email{imran$\_$aimphd@yahoo.com}
\address{Department of Mathematics \\ Yazd University, Yazd, Iran}
\email{davvaz@.yazd.ac.ir}



\begin{abstract}
Let $G=QD_{8k}~$ be the quasi-dihedral group of order $8n$ and $\theta$ be an automorphism of $QD_{8k}$ of finite order. The fixed-point set $H$ of $\theta$ is defined as $H_{\theta}=G^{\theta}=\{x\in G \mid \theta(x)=x\}$ and  generalized symmetric space $Q$ of $\theta$  is given by $Q_{\theta}=\{g\in G \mid g=x\theta(x)^{-1}~\mbox{for some}~x\in G\}.$

The characteristics of the sets $H$ and $Q$ have been calculated. It is shown that for any $H$ and $Q,~~H.Q\neq QD_{8k}.$ The $H$-orbits on $Q$ are obtained under different conditions. Moreover, the formula to find the order of $v$-th root of unity in $\mathbb{Z}_{2k}$ for $QD_{8k}$ has been calculated. The criteria to find the number of equivalence classes denoted by $C_{4k}$ of the involution automorphism has also been constructed. Finally, the set of twisted involutions $R=R_{\theta}=\{~x\in G~\mid~\theta(x)=x^{-1}\}$ has been explored.
\end{abstract}
\maketitle
\section{Introduction}

Serious work on groups generated by reflections began in the $19$th century. The finite subgroups of $O(3)$ generated by isometric reflections on the $2$-sphere (or equivalently, by orthogonal linear reflections on $\mathbb{R}^{3}$) had been determined by  M\"{o}bius in $1852$. In the second half of the $19$th century, work also began on finite reflection groups on $S^{n}$ for $n>2$ (or equivalently, finite linear reflection groups on $\mathbb{R}^{n+1}$). By the late $19'$s other groups generated by isometric reflections across the edges of polygons (with $3$ or more edges) in the hyperbolic plane had been studied by Klein and Poincar´e. The group of symmetries of such a root system was a finite group generated by reflections had been shown by Weyl in $1925$. This intimate connection with the classification of semi-simple Lie groups cemented reflection groups into a central place in mathematics. The two lines of research were united in (Coxeter, $1930$). He classified discrete groups generated by reflections on the $n$-dimensional sphere or Euclidean space. The notion of an abstract reflection group introduced by Jacques Tits which he called a "coxeter group". He considered pairs $(W,S)$ with $W$ a group and $S$ a set of involutions which generate $W$ so that the group has a presentation of the following form: the set of generators is $S$ and the set of relations is $\{(st)^{m(s,t)}\},$ where $m(s,t)$ denotes the order of $st$ and the relations range over all unordered pairs $ s,t\in S$ with $m(s,t)\neq \infty$. The pair $(W,S)$ is a coxeter system and $W$ is a coxeter group. This was demonstrated in (Bourbaki, $1968$; Bjorner and Francesco, $2005$).

The early work by Helminck on symmetric spaces concerned some algebraic results about their structure(Helminck, $1988$). Classification of the real symmetric spaces and their fine structure had been included in (Helminck, $1994$). By considering an open subgroup $H$ of the fixed point group $G^{\theta}$ of an involution $\theta$ of $G,$ the result was proved that {`}{`} For $k$ a field of characteristic not two, let $G$ be a connected semi-simple algebraic $k$-group and $\theta$ an involution of $G.$ Then $\theta$ is defined over $k$ if and only if $H^{\theta}$ is defined over $k$" (Helminck and Wang, $1993$). They also proved that {`}{`} if $G$ is a connected semi-simple algebraic group, $\theta_{i}$ involution of $G$ and $H_{i}=G^{\theta_{i}},~(i=1,2).$ Let $H_{1}^{\theta_{1}}$ and $H_{2}^{\theta_{2}}$ be the identity components of $H_{1}$ and $H_{2}$ respectively. If $H_{1}^{\theta_{1}}=H_{2}^{\theta_{2}}$ then $\theta_{1}=\theta_{2}."$

A mini-workshop on generalizations of symmetric spaces was organized by  Ralf Kohl at.el in $2012$. Mathematicians in this area brought together to present their current research projects or trigger new collaboration via comparison, analogy, transfer, generalization and unification of method in this workshop.

 The structure of automorphism group of dihedral group has been found by Cunningham at.el ($2012$). They also characterized the automorphism group with $\theta^{2}=\mbox{id}$ and calculated $H$ and $Q$ for each automorphism for the family of finite dihedral groups. They also proved that the set of involutions is partitioned in such a way that there are at most two distinct equivalence classes of involutions in each of the partition. In involution case, $Q$ is always a subgroup of $G.$ Moreover, the infinite dihedral group is also discussed.

Let $G = Dc_n$ be the dicyclic group of order $4n$. Let $\phi$ be an
automorphism of $G$ of order $k$. Bishop at.el ($2013$) describe $\phi$ and the generalized symmetric space $Q$ of $G$ associated with $\phi$. When $\phi$ is an involution, they describe its fixed point group $H$ along with the $H$-orbits and
$G$-orbits of Q corresponding to the action of $\phi$-twisted conjugation.

In this paper, we extent the study to quasi-dihedral group $QD_{8k}$.
\section{Preliminaries}
 The quasi-dihedral group is denoted by $QD_{8k},$ where $QD_{8k} = \langle a , b~~: a^{4k}=b^{2}=1 , ba=a^{2k-1}b\rangle.$
In reset of the whole paper, we denote it by $QD_{2m},$ where $m=4k$ and $n=\frac{m}{2}=2k.$ That is
$~~~~~~~~~~~~~~~~~~~~~~~QD_{2m} = \langle a , b~~: a^{m}=b^{2}=1 , ba=a^{n-1}b\rangle$.
It is clear from the presentation that $a^{i}b^{j}\in QD_{2m}$ is in normal form for some integer $0 \leq i < m$ and $ j \in \{ 0,1 \}.$ In the whole study, $\mathbb{Z}_{n}$ denotes the additive group of integers modulo $n$ and $U_{m}$ denotes the multiplicative group of units of $\mathbb{Z}_{n}$ and $r\in U_{m}$ if and only if $\gcd(r,m)=1.$ The $v$-th root of unity in $\mathbb{Z}_{n}$ which can be described by $R^{v}_{m}=\{r\in U_{m}~\mid r^{v}=1~(mod~n) \}.$

Let $G$ be a group, then an element $a$ is a conjugate with $b$ in $G$ if and only if there exists an element $x$ such that $a=xbx^{-1}$. For $a\in G,$ the function $\phi_{a}(x)=axa^{-1}~~\forall~~x\in G$ is called an inner automorphism of G. Let $\theta_{1},~\theta_{2}$ be two automorphisms, then $\theta_{1} \sim \theta_{2}$ if and only if they are conjugate to each other. If there exists another automorphism $\sigma$ such that $\sigma \theta_{1} \sigma^{-1}=\theta_{2}$.
\begin{definition} For any $c\in \mathbb{Z}_{n},$ we define
 $ZD\mbox{iv}~(c)=\{y\in \mathbb{Z}_{n}\mid cy\equiv0~(mod~n)\}$ and easy to see that
 $\mid ZD\mbox{iv}~(c)\mid =\gcd(c,n)$
\end{definition}
\begin{notn}
Let us fixed the following notations, which will be used in the whole paper.
\begin{itemize}
              \item For fixed $r\in R^{v}_{m}.$ The number of equivalence classes with fixed leading co-efficient $r$ is denoted by $N_{r}=\mid \{\bar{\theta}~\mid~\theta=rx+s\in Aut_{v}(D_{2m})\} \mid$
              \item Let us consider the subset of automorphism denoted by $Aut_{v}(QD_{2m})$ and is defined as  $Aut_{v}(QD_{2m})=\{\theta\in Aut(QD_{2m})\mid\theta^{v}=\mbox{id}~\}$.
            \end{itemize}
\end{notn}
The following lemma is easy to prove.
\begin{lem}\label{f1}
If $QD_{2m}$ be the quasi-dihedral group, then
\begin{enumerate}
\item $ba^{i}=a^{i(n-1)}b$  \ \ for all $i=0,1,\ldots,(m-1),$
\item $\mid a^{2i}b\mid=2$ and $\mid a^{2i+1}b\mid=4$ \ \ for all $i=0,1,\ldots,(n-1).$
\end{enumerate}
\end{lem}

\begin{lem}
The automorphism group of $QD_{2m}$ is isomorphic to the group of affine linear transformations of $\mathbb{Z}_n.$ In particular, we have
$$Aut(QD_{2m})\cong Aff(\mathbb{Z}_n)= \{rx+s: \mathbb{Z}_n \rightarrow \mathbb{Z}_n \mid r\in U_{m}, s\in  \mathbb{Z}_n \}$$
and the action of $rx+s$ on elements of $QD_{2m}$ in normal form is given by
\begin{equation}\label{e1}
(rx+s)(a^{j}b^{t}) = a^{rj+[(t-1)n-(t-2)]s}b^{t}
\end{equation}
\end{lem}

\begin{proof}
Let $\theta$ be an automorphism, then
$\theta(a)=a^{r}$ and
$\theta(b)=(a^{s}b),$ where $\gcd( r, m ) = 1$, $s=0, 2, 4,\ldots ,(m-2).$
  Consider arbitrary element $a^{j}b^{t}\in QD_{2m},$ and $\theta$ being homomorphism,
\begin{equation}\label{e2}
\theta(a^{j}b^{t})=\theta(a^{j}) \theta(b^{t}),
\end{equation}
then
$\theta(a^{j})=  a^{rj}$ and
$\theta(b^{t})= {(a^{s}b)(a^{s}b) (a^{s}b)\ldots(a^{s}b)}= {a^{s}(ba^{s})(ba^{s})b\ldots(a^{s}b)}.$

By Lemma \ref{f1}, we have\\
 $\theta(b^{t})= {a^{s}[a^{s(n-1)}b] (ba^{s})(ba^{s}) b\ldots(a^{s}b)}= {a^{s+s(n-1)}(b^{2}a^{s})(ba^{s})b\ldots(a^{s}b)}$\\
$\theta(b^{t})= {a^{sn}(b^{2}a^{s})(ba^{s})b\ldots(a^{s}b)}= {a^{sn}(a^{s(n-1)}b^{2})(ba^{s})b\ldots(a^{s}b)}$\\
$\theta(b^{t})= {a^{sn+sn-s}b^{3}(ba^{s})b\ldots(a^{s}b)}= {a^{2sn-s}b^{3}(ba^{s})b\ldots(a^{s}b)}$
after $t$-steps, we have
$\theta(b^{t})= a^{(t-1)ns-(t-2)s}b^{t}.$
Put values in equation \ref{e2}, we get\\
$\theta(a^{j}b^{t})=a^{rj}.a^{(t-1)ns-(t-2)s}b^{t}=a^{rj+(t-1)ns-(t-2)s}b^{t}=a^{rj+[(t-1)n-(t-2)]s}b^{t}.$
Since $b^{2}=1$, so $t\in\{0,1\}.$
If $t=0$, then
$\theta(a^{j})= a^{rj+(2-n)s}$  $~~~~~~~\Longrightarrow rx+(2-n)s\in Aff (\mathbb{Z}_{n}).$
If $t=1$, then
$\theta(a^{j}b)= a^{rj+s}b$ $~~~~~~~~~~\Longrightarrow rx+s\in Aff (\mathbb{Z}_{n}).$\\
\textbf{Conversely,}
For $ f\in Aff (\mathbb{Z}_{n}),$ that is $~f=rx+s$  where $r\in U_{m},~~s\in \mathbb{Z}_{n}$
and it is given as a group action
$f(a^{j}b^{t}) = a^{rj+[(t-1)n-(t-2)]s}b^{t}$ \ \ for all \ $a^{j}b^{t}\in QD_{2m}.$\\
Now, we will show that $f$ is an automorphism.
For this let us define $\theta :QD_{2m}\longrightarrow QD_{2m}$ by
$\theta(a^{j}b^{t})=f(a^{j}b^{t}) = a^{rj+[(t-1)n-(t-2)]s}b^{t}.$
We will show that $\theta$ is a bijective homomorphism. For this we will show that $\theta$ is one to one and onto.\\
\textbf{one to one:}
Suppose that
$\theta(a^{j_{1}}b^{t_{1}})=\theta(a^{j_{2}}b^{t_{2}})$, then
$ a^{rj_{1}+[(t_{1}-1)n-(t_{1}-2)]s}b^{t_{1}}= a^{rj_{2}+[(t_{2}-1)n-(t_{2}-2)]s}b^{t_{2}}$
$ \Rightarrow a^{rj_{1}+[(t_{1}-1)n-(t_{1}-2)]s-rj_{2}-[(t_{2}-1)n-(t_{2}-2)]s}b^{t_{1}-t_{2}}=e.$\\
 \textit{Case$~1:$}\\
$rj_{1}+[(t_{1}-1)n-(t_{1}-2)]s-rj_{2}-[(t_{2}-1)n-(t_{2}-2)]s=0$        and        $t_{1}-t_{2}=2$,
since, $t\in\{0,1\}$    therefore  $t_{1}-t_{2}\neq2$ not possible.\\
\textit{Case$~2:$}\\
$rj_{1}+[(t_{1}-1)n-(t_{1}-2)]s-rj_{2}-[(t_{2}-1)n-(t_{2}-2)]s=m$        and        $t_{1}-t_{2}=0$
$\Rightarrow r(j_{1}-j_{2})+ns(t_{1}-t_{2})-s(t_{1}-t_{2})=m$
$\Rightarrow r(j_{1}-j_{2})=m$ $\Rightarrow r\mid m $ but $\gcd(r,m)=1$, not possible.  Now
$ r(j_{1}-j_{2})=m$ will be only possible when $j_{1}-j_{2}=0,~~r\neq 0$
$j_{1}=j_{2}~~~~~\Rightarrow a^{j_{1}}=a^{j_{2}}~~~~~\Rightarrow a^{j_{1}}b^{t_{1}}=a^{j_{2}}b^{t_{2}}$         $\Rightarrow \theta$ is one to one.\\
\textbf{onto:}
It is cleared from the definition of the group action that for every $a^{rj+[(t-1)n-(t-2)]s}b^{t}$ there exists $a^{j}b^{t}\in QD_{2m}$ such that $\theta(a^{j}b^{t})=a^{rj+[(t-1)n-(t-2)]s}b^{t}$ $~~\Rightarrow \theta$ is onto. $ \mbox{Hence}\ \ \ Aut(QD_{2m})\cong Aff(\mathbb{Z}_{n}).$
\end{proof}
\section{Structure of Automorphisms of  $QD_{2m}$}
\begin{prop}\label{f2}
For any integer $v\geq1$ and $\theta=rx+s$ be an automorphism of $QD_{2m}$ then $\theta^{v}=$id if and only if $r\in R^{v}_{m}$ and $s\in ZD\mbox{iv}(r^{v-1}+r^{v-2}+\cdots +r+1)$.
\end{prop}
\begin{proof}
We will use the mathematical induction on $v$ to show first the following statement:
$\theta^{v}(x)=r^{v}x+(r^{v-1}+r^{v-2}+\cdots +r+1)s$.

For $v=2,~3$
$\theta^{2}(x)=\theta[\theta(x)]=\theta(rx+s)=r(rx+s)+s=r^{2}x+rs+s=r^{2}x+(r+1)s$.
Similarly,
          $\theta^{3}(x)=r^{3}x+r^{2}s+rs+s=r^{3}x+(r^{2}+r+1)s$.

Suppose that the statement is true for $v=i$ that is\\
$\theta^{i}(x)=r^{i}x+(r^{i-1}+r^{i-2}+\cdots +r+1)s$.
We will show that the statement is true for $v=i+1$.
Take $\theta^{i+1}(x)=\theta[\theta^{i}(x)]=\theta[r^{i}x+(r^{i-1}+r^{i-2}+\cdots +r+1)s]$
$\theta^{i+1}(x)=r[r^{i}x+(r^{i-1}+r^{i-2}+\cdots +r+1)s]+s$
$=r^{i+1}x+(r^{i}+r^{i-1}+\cdots +r+1)s]$.
          Hence , the statement is true for all positive values of $v.$
We have identified the identity automorphism with $x\in Aff(\mathbb{Z}_{n}),$  thus
$r^{v}\equiv1~~~~~(mod~n)$ and $(r^{v-1}+r^{v-2}+\cdots +r+1)s\equiv0~~~~(mod~n)$
$\Leftrightarrow r\in R^{v}_{m}$ and $s\in ZD\mbox{iv}(r^{v-1}+r^{v-2}+\cdots +r+1)$ as required.
\end{proof}

\begin{expl}
Let $G=QD_{56}$ and $v=3$. In this case $R_{28}^{3}=\{ 1,9,11,15,23,25 \},$
take $r=9$, then $ZDiv(9^{2}+9+1)=ZDiv(91)=\{0,2,4,6,8,10,12\}$.
Now, consider $ \theta_{1}=9x+6 \in Aut(QD_{56}),$ then after simple calculation we get $\theta_{1} ^{3}=x=$id. Similarly, for $ \theta_{2}=9x+12 \in Aut(QD_{56}),$ we have $\theta_{2} ^{3}=x=$id.
\end{expl}
\begin{prop} \label{f3}
For any $m$ and $v\geq1,$ we have
$$|Aut_{v}(QD_{2m})|=\sum_{r\in R^{v}_{m}}\gcd(r^{v-1}+r^{v-2}+\cdots +r+1,n)$$
\end{prop}
\begin{proof}
This follows from above Proposition: for any ${r\in R^{v}_{m}}$ there are $$|ZD\mbox{iv}(r^{v-1}+r^{v-2}+\cdots +r+1)|=\gcd(r^{v-1}+r^{v-2}+\cdots +r+1,n)$$ element $s$ such that $(r^{v-1}+r^{v-2}+\cdots +r+1)s\equiv0~(\mbox{mod~}n)$ and every automorphism $\theta\in Aut_{v}(QD_{2m})$ must be of this form.
\end{proof}

\begin{prop} \label{f11}
$Let~\theta_{1}=r_{1}x+s_{1}~and~\theta_{2}=r_{2}x+s_{2}$ be two different automorphism of $QD_{2m}$. Then $\theta_{1}\sim\theta_{2}$ if and only if $r_{1}=r_{2}$ and $fs_{1}-s_{2}\in \langle r_{1}-1 \rangle \leq \mathbb{Z}_{n}$ for some $f\in U_{m}.$
\end{prop}

\begin{proof}
Given that  $~\theta_{1}=r_{1}x+s_{1}$ and $\theta_{2}=r_{2}x+s_{2}\in Aut(QD_{2m})$.
Suppose that $\sigma=fx+g$ be any arbitrary automorphism of $QD_{2m}$ and define
$\sigma^{-1}=f^{-1}x-f^{-1}g$.
 We will check that $\sigma \cdot\sigma^{-1}=x$.
  For this take
  $\sigma \cdot\sigma^{-1}=\sigma (\sigma^{-1}(x))=\sigma(f^{-1}x-f^{-1}g)=f(f^{-1}x-f^{-1}g)+g=x-g+g=x$.
  Now, we calculate
\begin{equation}\label{e3}
    \sigma \theta_{1}\sigma^{-1}=\sigma \cdot ~(\theta_{1} \cdot ~\sigma^{-1})
\end{equation}
Take
$\theta_{1} \cdot ~\sigma^{-1}=\theta_{1}[\sigma^{-1}(x)]=\theta_{1}[f^{-1}x-f^{-1}g]=r_{1}(f^{-1}x-f^{-1}g)+s_{1}$
$~~~~~~~~~=r_{1}f^{-1}x-r_{1}f^{-1}g+s_{1}~~~~~~~$ put in \ref{e3}, we get
$\sigma \theta_{1}\sigma^{-1}=\sigma(r_{1}f^{-1}x-r_{1}f^{-1}g+s_{1})=f(r_{1}f^{-1}x-r_{1}f^{-1}g+s_{1})+g$
$\sigma \theta_{1}\sigma^{-1}=r_{1}x-r_{1}g+fs_{1}+g=r_{1}x+fs_{1}-g(r_{1}-1)$.
Now, by definition
$\sigma \theta_{1}\sigma^{-1}=\theta_{2}$ if and only if $r_{1}=r_{2}$ and
$fs_{1}-g(r_{1}-1)=s_{2}$ $~~~~~~~~\Rightarrow fs_{1}-s_{2}=g(r_{1}-1)$
$\Rightarrow fs_{1}-s_{2}\in \langle r_{1}-1 \rangle.$
\end{proof}

\begin{expl}
Let $G=QD_{64}$ and consider $\theta_{1}=7x+3, ~\theta_{2}=7x+14,~\theta_{3}=7x+11$ three automorphisms of $QD_{64}$. Then it is easy to check that $\theta_{1}\nsim\theta_{2}$ because there does not exist such $f\in U_{32}$ for which $3f-14 \in \langle 6\rangle$. And $\theta_{1}\sim\theta_{3}$ because there exist some $f\in U_{32}$ for which $3f-14 \in \langle 6\rangle$.
\end{expl}

\begin{prop} \label{f15}
Suppose $m$ is fixed and let $v\geq1.$
\begin{enumerate}
\item For any $r\in R^{v}_{m}~;\langle r-1 \rangle \leq ZD\mbox{iv}(r^{v-1}+r^{v-2}+\cdots +r+1).$
\item For any $r\in R^{v}_{m};~U_{m}$ acts on the cosets $\frac{ZD\mbox{iv}(r^{v-1}+r^{v-2}+\cdots +r+1)}{\langle r-1\rangle}.$
\item The set $Aut_{v}(QD_{2m})$ is partitioned into equivalence classes indexed by pair $(r,A),$ where $r\in R^{v}_{m}$ and $A$ is an orbit of $U_{m}$ on $\frac{ZD\mbox{iv}(r^{v-1}+r^{v-2}+\cdots +r+1)}{\langle r-1\rangle}.$
\end{enumerate}
\end{prop}

\begin{proof}
 \begin{enumerate}
 \item We know that
 $(r-1).(r^{v-1}+r^{v-2}+\cdots +r+1)=r^{v}-1\equiv0~~(\mbox{mod~n})$
 $\Rightarrow(r-1).(r^{v-1}+r^{v-2}+\cdots +r+1)\equiv0~~(\mbox{mod~n})$
 $\Rightarrow(r-1)\in ZD\mbox{iv}(r^{v-1}+r^{v-2}+\cdots +r+1)$
 $\Rightarrow \langle r-1\rangle \leq ZD\mbox{iv}(r^{v-1}+r^{v-2}+\cdots +r+1).$
 \item We simply recognize that $U_{m}=Aut(\mathbb{Z}_{n})$ acts on $\mathbb{Z}_{n}$ by multiplication and since every subgroup is cyclic, $U_{m}$ must stabilize the subgroups of $\mathbb{Z}_{n}.$
 \item It is simply Proposition \ref{f11} in terms of the $U_{m}$-action on $\frac{ZD\mbox{iv}(r^{v-1}+r^{v-2}+\cdots +r+1)}{\langle r-1\rangle}.$
 \end{enumerate}
 \end{proof}
\begin{expl}
Let $G=QD_{40}$,and $v=2$. Then $U_{20}=\{ 1,3,7,9,11,13,17,19\}$ and $\mathbb{Z}_{10}=\{0,1,2,\ldots,9\}.$
\begin{enumerate}
  \item Since $v=2$, so $R_{20}^{2}=\{ 1,9,11,19\}.$ Take $r=9$ then $\langle9-1\rangle=\langle8\rangle=\{0,2,4,6,8\}$ and $ZDiv(10)=\{ 0,1,2,\ldots ,9\}.$ So, $\{ 0, 2, 4, 6, 8\}\subseteq \{ 0,1,2,\ldots ,9\}$
  \item For $U_{20}$ action on the cosets $\frac{ZDiv(r+1)}{\langle r-1\rangle}.$ Take $r=9$ then
  $\frac{ZDiv(r+1)}{\langle r-1\rangle}=\frac{ZDiv(10)}{\langle 8\rangle}=\frac{E=\{ 0,1,2,\ldots ,9\}}{F=\{0,2,4,6,8\}}=x+F~\forall~x\in E$ provide two sets $F$ and $F_{1}=\{1,3,5,7,9\}$.
  Now, the action of $U_{20}$ on the cosets $\frac{ZDiv(10)}{\langle 8\rangle}$ has two sets.
  $\frac{ZDiv(10)}{\langle 8\rangle}= y.F=F~\forall~y \in U_{20}$ and $\frac{ZDiv(10)}{\langle 8\rangle}=y.F_{1}=F_{1}~\forall~y \in U_{20}$
  \item The action of $U_{20}$ on $\frac{ZDiv(10)}{\langle 8\rangle}=F, F_{1}.$ Then the number of equivalence classes of $Aut(QD_{40})$ indexed by the pair $(r,~A)$ where $r \in R^{2}_{20}$ and $A$ is an orbit of $U_{20}$ on $\frac{ZDiv(10)}{\langle 8\rangle}$ is 8.
\end{enumerate}
\end{expl}
\begin{thm} \label{f16}
 If $m$ is a fixed positive integer and $r\in R^{v}_{m},$ then the number of orbits of $U_{m}$ on $\frac{ZD\mbox{iv}(r^{v-1}+r^{v-2}+\cdots +r+1)}{\langle r-1\rangle}$ is equal to the number of divisors of $\frac{\gcd(r-1,n).\gcd(r^{v-1}+r^{v-2}+\cdots +r+1,n)}{n}.$
\end{thm}
\begin{proof}
The $U_{m}$-orbit on $\mathbb{Z}_{n}$ are indexed by the subgroups of  $\mathbb{Z}_{n}.$ Thus the $U_{m}$-orbit
on $\frac{ZD\mbox{iv}(r^{v-1}+r^{v-2}+\cdots +r+1)}{\langle r-1\rangle}$ are indexed by subgroups $L\leq\mathbb{Z}_{n}$ such that $$\langle r-1\rangle \leq L \leq ZD\mbox{iv}(r^{v-1}+r^{v-2}+\cdots +r+1).$$ It is well known that the subgroup lattice of $\mathbb{Z}_{n}$ is isomorphic to the divisor lattice of $n.$ The subgroup $\langle r-1\rangle$ corresponds to the divisor $\gcd(r-1,n)$ and the subgroup $ZD\mbox{iv}(r^{v-1}+r^{v-2}+\cdots +r+1)$ correspond to the divisor $\frac{n}{\gcd(r^{v-1}+r^{v-2}+\cdots +r+1,n)},$ and the subgroups between these groups correspond to the divisor of $n$ between $\gcd(r-1,n)$ and $\frac{n}{\gcd(r^{v-1}+r^{v-2}+\cdots +r+1,n)}$ in the divisor lattice. Finally, it is known that this sub-lattice of the divisor of $n$ is isomorphic to the divisor lattice of $$\frac{\gcd(r-1,n)}{\frac{n}{\gcd(r^{v-1}+r^{v-2}+\cdots +r+1,n)}}=\frac{\gcd(r-1,n).\gcd(r^{v-1}+r^{v-2}+\cdots +r+1,n)}{n}.$$
\end{proof}

\section{Fixed Ssts and Symmetric Spaces of Automorphisms}
In this section if $\theta$ is an automorphism the fixed-point set $H_{\theta}$ and the symmetric space $Q_{\theta}$ will be discussed, where $H_{\theta}=\{x\in G \mid \theta(x)=x\}$ and $Q_{\theta}=\{g\in G \mid g=x\theta(x)^{-1}~\mbox{for some}~x\in G\}$. When $\theta$ is understood to be fixed the subscript from our notation has been dropped.

The following theorem characterizes the sets $H$ and $Q$ in the case of quasi-dihedral groups.

\begin{thm} \label{f4}
If $\theta=rx+s$ is an automorphism of $G=QD_{2m}$ of finite order, then
$H=\{a^{j}\mid j(r+2s-1)\equiv0~(mod~n)\}\cup\{a^{j^{'}}b\mid j^{'}(r+2s-1)\equiv-s~(mod~n)\}$
and $Q=\{a^{j}\mid j\in \langle r+2s-1\rangle\cup [l(1-r)-s]~(mod~n)\}$
where,$~~j^{'}$ and $l$  are even numbers.
\end{thm}
\begin{proof}
Since, $\theta$ is an automorphism and
$\theta(a^{j}b^{t})=a^{rj+[(t-1)n-(t-2)]s}b^{t}.$
First we will find set the $H.$ Let us consider $a^{j}\in G,$ then by definition of $H$, we have $\theta(a^{j})=a^{j},$
$ \theta(a^{j})= a^{[r+(2-n)s]j}=a^{j},$
$\Rightarrow n\mid j(r+2s-1).$

Now consider $a^{j^{'}}b\in G,$ then $ \theta(a^{j^{'}}b)=a^{j^{'}}b$ so,
$a^{j^{'}}b=a^{[r+(2-n)s]j^{'}}a^{s}b,$
thus $ n\mid(r+2s-1)j^{'}+s.$
Hence, $H=\{a^{j}\mid j(r+2s-1)\equiv0~(\mbox{mod~n})\}\cup\{a^{j^{'}}b\mid j^{'}(r+2s-1)\equiv-s~(\mbox{mod~n})\},$
where $j^{'}$ is even number.

Now we will find the set $Q.$
L $a^{j} \in G$ be an arbitrary element and let $a^{i}$ be a fixed element of $G$, where $j \neq i$, then by definition of $Q$,
$a^{j}=a^{i} \theta (a^{i})^{-1}=a^{i}\theta(a^{-i})=a^{i}.a^{-[r+(2-n)s]i}=a^{i[1-r-(2-n)s]}$
$\Rightarrow j\equiv i[1-(r+2s)]~~(\mbox{mod~n})$ $~~~~~~\Rightarrow j\in \langle r+2s-1\rangle.$

Now consider $a^{l}b\in G,~~~~~$ where $l$ is even number, then
$a^{j}=(a^{l}b)\theta(a^{l}b)^{-1}=(a^{l}b)\theta(a^{l}b)=(a^{l}b)(a^{rl+s}b)=a^{l}(ba^{rl+s})b=a^{l}a^{(rl+s)(n-1)}bb=a^{l+(rl+s)(n-1)}$
            $\Rightarrow j = l+(rl+s)(n-1)$   $\Rightarrow j=l[1+r(n-1)]+s(n-1)$ $~~~~\Rightarrow j\equiv l(1-r)-s~~~(\mbox{mod~n}).$

            So, $Q=\{a^{j}\mid j\in \langle r+2s-1\rangle\cup [l(1-r)-s]~(\mbox{mod~n})\},$  where $~l$ is an even numbers.
\end{proof}
\begin{expl}
Let us consider $G=QD_{64}$ and $\theta=3x+12 \in Aut(D_{64}),$ then
$H=\{a^{j}\mid 10j\equiv0~(mod~16)\}\cup\{a^{j^{'}}b\mid 10j^{'}\equiv-4~(mod~16)\}=\{1,a^{8},a^{2}b,a^{10}b\}$ and
$Q=\{a^{j}\mid j\in \langle 10\rangle\cup (14j^{'}+4)~(mod~16)\}=\{1,a^{2},a^{4},\ldots,a^{14}\}$
\end{expl}
            Using the descriptions of $H$ and $Q$ , we obtain further results as well.

\begin{cor} \label{f5}
If $\theta=rx+s$ is an automorphism of $QD_{2m}$ of finite order, where $m=4k,~n=2k$ and $k$ is even, then $H$ will be cyclic or non-cyclic according to the following conditions
\begin{enumerate}
 \item Let $\gcd(r+2s-1,~n)=d_{1}$. If $s=2w+1$ is odd and $d_{1}\nmid w$ then $H$ will be cyclic that is
 $H \cong ZD\mbox{iv}(r+2s-1)$ ~~~~~\mbox{and}~~ $\mid H \mid =d_{1}$
 \item Let $\gcd(r+2s-1,~k)=d_{2}$. If  $s=2w$ is even and $d_{2}\mid w$  then we have two cases
 \begin{description}
 \item[a)] If $w$ is odd, then $H$ will be cyclic
 $\mid H \mid =d_{1}$ $~~~~~~\mbox{and}~~H \cong ZD\mbox{iv}(r+2s-1).$
\item[b)] If $w$ is even
\begin{description}
\item[i)] If $d_{2} \nmid w$ then $H$ will be cyclic
$\mid H \mid =d_{1}$ $~~~~\mbox{and}~~H \cong ZD\mbox{iv}(r+2s-1).$
\item[ii)] If $d_{2}\mid w$ then $H$ will be non-cyclic. We have two cases
\begin{description}
\item[I-)] if $4\nmid k$ then
$\mid H \mid =\frac{3}{2}d_{1}$
$H=\{\langle a^{\frac{k}{2}}\rangle ,a^{2l}b,~a^{2q}b \mid l\neq q~ \mbox{for some}~l,q\in \{1,2,...k\}\}.$
It is not a subgroup of $G.$
\item[II-)] if $4\mid k$ then
$\mid H \mid =2d_{1}$  $~\mbox{and}~~H \cong ZD\mbox{iv}(r+2s-1)\times Z_{2}.$
\end{description}
\end{description}
\end{description}
\end{enumerate}
\end{cor}
\begin{expl}
Let $G=QD_{48},$ then $\gcd(r,m)=1$ $~\Rightarrow r\in U_{24}=\{ 1,5,7,11,13,17,19,23\}$ and $s \in \mathbb{Z}_{12}=\{ 0,1,2,\ldots ,11\}.$
\begin{enumerate}
\item Consider  $\theta_{1}=7x+11 \in Aut(QD_{48})$ then $r=7,~s=2(5)+1=11$ $~~\Rightarrow r+2s-1=4$ and $d_{1}=\gcd(28,12)=4$ $\Rightarrow 4\nmid5$ then $H_{1}=\{1,a^{3},a^{6},a^{9}\}$ is cyclic. and $ZDiv4=\{0,3,6,9\}$ $\Rightarrow H_{1}\cong ZDiv4$ and $ O(H_{1})=d_{1}=4$
\item \begin{description}
\item[a)] Take  $\theta_{2}=11x+10\in Aut(QD_{48})$ then $r=11,~s=2(5)=10$
$\Rightarrow r+2s-1=6$ and $d_{1}=\gcd(6,12)=6$, $d_{2}=\gcd(6,6)=6$
then $H_{2}=\{1,a^{2},a^{4},\ldots, a^{10}\}$ is cyclic and $ZDiv6=\{0,2,4,6,8,10\}$
$\Rightarrow H_{2}\cong ZDiv6$ and $ O(H_{2})=d_{1}=6$
\item[b)] \begin{description}
\item[i)] For  $\theta_{3}=11x+4\in Aut(QD_{48})$ then $r=11,~s=2(2)=4$
$\Rightarrow r+2s-1=6$ and $d_{1}=\gcd(6,12)=6$, $d_{2}=\gcd(6,6)=6$ $\Rightarrow 6\nmid -2$ then $H_{3}=\{1,a^{2},a^{4},\ldots, a^{10}\}$ is cyclic and $ZDiv6=\{0,2,4,6,8,10\}$ $\Rightarrow H_{3}\cong ZDiv6$ and $ O(H_{3})=d_{1}=6$
\item[ii)] \begin{description}
\item[I-] For $\theta_{4}=5x+8\in Aut(QD_{48})$ then $r=5,~s=2(4)=8$ $~~\Rightarrow r+2s-1=8$ and $d_{1}=\gcd(8,12)=4$, $d_{2}=\gcd(8,6)=2$ $\Rightarrow 2\mid -4$ then $H_{4}=\{1,a^{3},a^{6},a^{9},a^{2}b,a^{8}b\}$ is non-cyclic and $ O(H_{4})=\frac{3}{2}.d_{1}=6$
\item[II-] Take $\theta_{5}=7x+4\in Aut(QD_{48})$ then $r=7,~s=2(2)=4$ $~~\Rightarrow r+2s-1=2$ and $d_{1}=\gcd(2,12)=2$, $d_{2}=\gcd(2,6)=2$ $\Rightarrow 2\mid -2$ then  $H_{5}=\{1,a^{6},a^{4}b,a^{10}b\}$ is non-cyclic and
 $O(H_{5})=2.d_{1}=4$
 \end{description}
 \end{description}
 \end{description}
\end{enumerate}
\end{expl}

\begin{cor} \label{f6}
If $m=4k,~n=2k$ and $k$ is odd, then $H$ will be cyclic or non-cyclic according to the following conditions
\begin{enumerate}
\item Let $\gcd(r+2s-1,~n)=d_{1}$. If $s=2w+1$ is odd then $H$ will be cyclic
$\mid H \mid =d_{1}$ $~~~~~~\mbox{and}~~H \cong ZD\mbox{iv}(r+2s-1).$
\item Let $\gcd(r+2s-1,~n)=d_{1}$ . If $s=2w$ is even then
\begin{description}
\item[a)] If $d_{1}=n$ then $H$ will be cyclic
$\mid H \mid =d_{1}=n$ $~~~~~~\mbox{and}~~H \cong ZD\mbox{iv}(r+2s-1).$
\item[b)] If $d_{1}=2$ then $H$ will be non-cyclic
$~\mid H \mid =\frac{3}{2} d_{1}=\frac{3}{2}~(2)=3$
and $H=\{1,~a^{k},~a^{2l}b\}~\mbox{for some}~l\in \{1,2,...,k\}$
is not a subgroup of $G.$
\end{description}
\end{enumerate}
\end{cor}
\begin{expl}
Let $G=QD_{40}$ where $m=20$, $n=10$, $k=5$. Since $\gcd(r,m)=1$ $~\Rightarrow r\in U_{20}=\{ 1,3,7,9,11,13,17,19\}$ and $s \in \mathbb{Z}_{10}=\{ 0,1,2,\ldots ,9\}.$
\begin{enumerate}
  \item Consider $\theta_{1}=3x+5 \in Aut(QD_{40})$ then $r=3,~s=5$ $~~\Rightarrow r+2s-1=2$ and $\gcd(2,10)=2$ then $H_{1}=\{1,a^{5}\}$ is cyclic and $ZDiv2=\{0,5\}$  $\Rightarrow H_{1}\cong ZDiv2$ and $\mid H_1 \mid =2$
  \item \begin{description}
          \item[a)] For  $\theta_{2}=7x+2\in Aut(QD_{40})$ then $r=7,~s=2$ $~~\Rightarrow r+2s-1=0$ and $\gcd(0,10)=10$ then $H_{2}=\{1,a^{2},a^{3},\ldots,a^{9}\}$ is cyclic and $ZDiv0=\{0,1,2,\ldots,9\}$ $\Rightarrow H_{2}\cong ZDiv0$ and $ \mid H_2 \mid=10$
          \item[b)] Take  $\theta_{3}=7x+4\in Aut(QD_{40})$ then $r=7,~s=4$ $~~\Rightarrow r+2s-1=4$ and $\gcd(4,10)=2$ then $H_{3}=\{1,a^{5},a^{4}b\}$ is non-cyclic and $\mid H_3 \mid =\frac{3}{2}.2=3$
        \end{description}
\end{enumerate}
\end{expl}
\begin{cor} \label{f7}
If $\theta=rx+s\in Aut_{v}(QD_{2m})$ be of finite order, where $m=4k,~n=2k$ and $k$ is even number, then $Q$ will be cyclic or non-cyclic according to the following conditions
\begin{enumerate}
  \item If $s=2w+1$ is odd, then $Q$ will be non-cyclic.
  \item Let$~~\gcd(r+2s-1,~n)=d_{1}$. If $s=2w$ is even, then
  \begin{description}
  \item[a)] If $d_{1}\mid(1-r)j^{'}-s~~~~~~~$ and
  \begin{description}
  \item[i)] if $r+2s-1\not\equiv k~~~~~(mod~n)~~~~$ then $Q$ will be cyclic.
  \item[ii)] If $r+2s-1\equiv k~~~~~(mod~n)$ we have two cases
\begin{itemize}
\item if $1-r\equiv k~~~~(mod~n)~~~~~$ then $Q$ will be cyclic.
\item If $1-r\not\equiv k~~~~~(mod~n)~~~~$ then $Q$ will be non-cyclic.
\end{itemize}
\end{description}
\item[b)] If $d_{1}\nmid (1-r)j^{'}-s~~~~~~~$ and
\begin{description}
 \item[i)] if $-s \not\equiv k~~~~~(mod~n)~~~~$ then $Q$ will be non-cyclic.
 \item[ii)] If $-s \equiv k~~~~~(mod~n)~~~~$ then $Q$ will be cyclic.
 \end{description}
 \end{description}
\end{enumerate}
\end{cor}

\begin{expl}
Let $G=QD_{48}$ where $m=24$, $n=12$, $k=6$. Since $\gcd(r,m)=1$ $~\Rightarrow r\in U_{24}=\{ 1,5,7,11,13,17,19,23\}$ and $s \in \mathbb{Z}_{12}=\{ 0,1,2,\ldots ,11\}.$
\begin{enumerate}
 \item Consider $\theta_{1}=7x+11 \in Aut(QD_{48})$ then $r=7,$ $s=11$ $ \Rightarrow r+2s-1=4$
then $Q_{1}=\{1,~a,~a^{4},~a^{8}\}$ is non-cyclic.
\item \begin{description}
\item[a)] \begin{description}
\item[i)] For $\theta_{2}=5x+8\in Aut(QD_{48})$ then $r=5$, $s=8$
$ \Rightarrow r+2s-1=8 \not\equiv k$ and $d_{1}=\gcd(8,~12)=4$,
$(1-r)j^{'}-s=-4j^{'}-8$ at $j^{'}=2~ yeilds~~4\mid -16$
then $Q_{2}=\{1,~a^{4},~a^{8}\}$ is cyclic.
\item[ii)] Take $\theta_{3}=7x+6\in Aut(QD_{48})$ then $r=7$, $s=6$
$ \Rightarrow r+2s-1=6 \equiv k$ and $d_{1}=\gcd(6,~12)=6$,
$(1-r)j^{'}-s=-6j^{'}-6$ at $j^{'}=2~ yeilds~~~6\mid -18$
and $1-r=-6\equiv6\equiv k$ then $Q_{3}=\{1,a^{6}\}$ is cyclic.
And for $\theta_{4}=11x+4\in Aut(QD_{48})$ then $r=11$, $s=4$
$ \Rightarrow r+2s-1=6 \equiv k$ and $d_{1}=\gcd(6,~12)=6$,
$(1-r)j^{'}-s=-10j^{'}-4$ at $j^{'}=2~ yeilds~~~6\mid -24$
and $1-r=-10\equiv 2 \not\equiv k$ then $Q_{4}=\{1,~a^{4},~a^{6},~a^{8}\}$ is non-cyclic.
\end{description}
\item[b)] \begin{description}
\item[i)] Consider $\theta_{5}=x+10\in Aut(QD_{48})$ then $r=1$, $s=10$
$ \Rightarrow r+2s-1=8$ and $d_{1}=\gcd(8,~12)=4$,
$(1-r)j^{'}-s=0j^{'}-10\equiv2~~~\Rightarrow4\nmid 2$ and $-s=-10\equiv 2 \not\equiv k$
then $Q_{5}=\{1,~a^{2},~a^{4},~a^{8}\}$ is non-cyclic.
\item[ii)] For $\theta_{6}=5x+6\in Aut(QD_{48})$ then $r=5$, $s=6$
$ \Rightarrow r+2s-1=4$ and $d_{1}=\gcd(4,~12)=4$,
$(1-r)j^{'}-s=-4j^{'}-6$ for any $j^{'},$ $ ~d_{1}\nmid -4j^{'}-6$
and $-s=-6\equiv6\equiv k$ then $Q_{6}=\{1,~a^{2},~a^{4},\ldots,a^{10}\}$ is cyclic.
\end{description}
\end{description}
\end{enumerate}
\end{expl}

\begin{cor} \label{f8}
If $m=4k,~n=2k$ and $k$ is odd number, then $Q$ will be cyclic or non-cyclic according to the following conditions
\begin{enumerate}
  \item If $s=2w$ is even then $Q$ will be cyclic.
  \item If $s=2w+1$ is odd
  \begin{description}
  \item[a)] $Q$ will be non-cyclic
 if either $r+2s-1\equiv 0~~~(mod~n)~~$ or $~~~~1-r\equiv 0~~(mod~n).$
 \item[b)] $Q$ will be cyclic
 \begin{description}
 \item[i)] if $1-r \equiv 0~~\mbox{and}~~s\equiv k~~~~~(mod~n).$
 \item[ii)] if $r+2s-1 \not\equiv 0~~\mbox{and}~~1-r \not\equiv 0~~~~~(mod~n).$
 \end{description}
 \end{description}
\end{enumerate}
\end{cor}
\begin{expl}
Let $G=QD_{40}$ where $m=20$, $n=10$, $k=5$. Since $\gcd(r,m)=1$ $~\Rightarrow r\in U_{20}=\{ 1,3,7,9,11,13,17,19\}$ and $s \in \mathbb{Z}_{10}=\{ 0,1,2,\ldots ,9\}.$
\begin{enumerate}
\item For $\theta_{1}=7x+6\in Aut(QD_{40})$ then $r=7$, $s=6$ $ \Rightarrow r+2s-1=8$
then $Q_{1}=\{1,a^{2},a^{4},a^{6},a^{8}\}$ is cyclic.
\item \begin{description}
\item[a)]Take $\theta_{2}=x+7\in Aut(QD_{40})$ then $r=1$, $s=7$ $\Rightarrow r+2s-1=4~and~~1-r=0$ then $Q_{2}=\{1,a^{2},a^{3},a^{4},a^{6},a^{8}\}$ is non-cyclic.
\item[b)]
\begin{description}
\item[i)] Consider $\theta_{3}=11x+5\in Aut(QD_{40})$ then $r=11$, $s=5$
$ \Rightarrow r+2s-1=0,~1-r=0,~~s=5\equiv k$ then
$Q_{3}=\{1,a^{5}\}$ is cyclic.
\item[ii)] For $\theta_{4}=7x+5\in Aut(QD_{40})$ then $r=7$, $s=5$
$ \Rightarrow r+2s-1=6,~~1-r=-6 \not\equiv 0$ then $Q_{4}=\{1,a,a^{2},\ldots,a^{9}\}$ is cyclic.
\end{description}
\end{description}
\end{enumerate}
\end{expl}

\begin{cor} \label{f9}
Let $\theta=rx+s\in Aut_{v}(QD_{2m})$ be of finite order, where $m=4k,~n=2k$ and $k$ is even number. Then $Q$ is isomorphic to a cyclic group as follows:
\begin{enumerate}
\item If$~~~~~~\mid Q \mid=2~~~$ then
$~~~~~~~~~~~~~Q= \langle a^{k}~:~a^{n}=1 \rangle$  $~~\mbox{and}~~~~~~Q\cong \langle k \rangle.$
\item If $~~\mid Q \mid=\frac{k}{2}~$ then
\begin{description}
\item[a)] Let $\gcd(r+2s-1,~n)=d_{1}$. If $n\nmid r+2s-1$ then $Q= \langle a^{d_{1}}~:~a^{n}=1 \rangle$ and$~~~~~~Q\cong \langle d_{1} \rangle.$
\item[b)] If $n\mid r+2s-1 ~~\mbox{and}~~\gcd(r+2s-1,~n)=n$
$~~~~~~~~~~\mbox{then}~~~Q= \langle a^{4}~:~a^{n}=1 \rangle$ and$~~~~~~Q\cong \langle 4 \rangle.$
\end{description}
\item If$~~~~~~\mid Q \mid=k~~~\mbox{and}~~~~\gcd(r+2s-1,~n)=d_{1}$ then
\begin{description}
\item[a)] if $(r+2s-1)\cap (1-r)j^{'}-s=\emptyset $
$~~~~~~~~~~\mbox{then}~~~Q= \langle a^{d_{1}/2}~:~a^{n}=1 \rangle$ and$~~~~~~Q\cong \langle \frac{d_{1}}{2} \rangle.$
\item[b)] if $(r+2s-1)\cap (1-r)j^{'}-s\neq \emptyset $
$~~~~~~~~~~\mbox{then}~~~Q= \langle a^{d_{1}}~:~a^{n}=1 \rangle$ and$~~~~~~Q\cong \langle d_{1} \rangle.$
\end{description}
\end{enumerate}
\end{cor}

\begin{cor} \label{f10}
Let $\theta=rx+s\in Aut_{v}(QD_{2m})$ be of finite order, if $k$ is odd number. Then the conditions for $Q$ to be cyclic and isomorphism relationship are:
\begin{enumerate}
  \item If $s$ is even, then
\begin{description}
  \item[a)] if $~r=1,~s=0$ $~\mid Q \mid=1$ then $Q=\{1\}$ $~~\mbox{and}~~Q\cong \langle 0 \rangle.$
  \item[b)] if $~\mid Q \mid=k$ then $~~Q= \langle a^{2}~:~a^{n}=1 \rangle\cong \langle 2 \rangle.$
\end{description}
  \item If $s$ is odd, then\begin{description}
  \item[a)] Let $n\mid 1-r, \  n\mid s-k$.
If $~~\mid Q \mid=2$ then $~Q= \langle a^{k}~:~a^{n}=1 \rangle\cong \langle k \rangle.$
\item[b)] if $r+2s-1 \not\equiv 0~~\mbox{and}~~1-r \not\equiv 0~~~~~(mod~n)$
and$~~\mid Q \mid=n$ then $~Q= \langle a~~:~a^{n}=1 \rangle\cong \langle 1 \rangle.$
\end{description}
\end{enumerate}
\end{cor}

\begin{cor}
Let $\theta=rx+s\in Aut(QD_{2m})$ be a fixed automorphism. Then for any $H$ and $Q,$ $HQ\neq QD_{2m}.$
\end{cor}
\begin{exple}\label{f21}
Consider the involution $\theta_{1}=17x+8$ of $QD_{48}$. Then $H_{1}=\{1,a^{3},a^{6},a^{9},a^{2}b,a^{8}b\}~~~~$ and $~~~Q_{1}=\{1,a^{4},a^{8}\}$
$\Rightarrow H_{1}Q_{1}\neq QD_{48}.$
Now, consider the automorphism $\theta_{2}=5x+3$ of order $4.$ Then $H_{2}=\{1,a^{6}b\}~~~~$ and $~~~Q_{2}=\{1,a^{2},a^{4},\ldots ,a^{10},a,a^{5},a^{9}\}$ $\Rightarrow H_{2}Q_{2}\neq QD_{48}.$
\end{exple}
\begin{prop}\label{f20}
Let $\theta=rx+s\in Aut_{v}(QD_{2m})$ and for any $a^{l}\in Q.$
\begin{description}
  \item[i)] If $l$ is odd, then
$H \backslash Q=\{\{a^{j}\}~\mid j\in \langle r+2s-1\rangle\cup [(1-r)j^{'}-s]~(mod~n)\}$
\item[ii)] If $l$ is even then
$H \backslash Q=\{\{a^{j},~a^{-j}\}~\mid j\in \langle r+2s-1\rangle\cup [(1-r)j^{'}-s]~(mod~n)\}$
\end{description}
where, $~j^{'}$ is any even number. In either case $G \backslash Q=\{Q\}$, that is there is only a single $G$-orbit on $Q.$
\end{prop}
\begin{proof}
Since, $H$ is fixed by $\theta $
thus the action of $H$ on $Q$ is simply by conjugation. Notice that
$Q \subseteq \langle a \rangle \leq QD_{2m}$
so, the action of $H$ on general element $a^{l}$ of $Q.$
Let $~a^{i}\in H$
then, $a^{i}.a^{l}.a^{-i}=a^{i+l-i}=a^{l}.$
So, $\langle a \rangle$ fixes $Q$ point wise.

Suppose $~a^{i^{'}}b\in H$, then
$(a^{i^{'}}b).a^{l}.(a^{i^{'}}b)^{-1}=a^{i^{'}}(ba^{l})b^{-1}a^{-i^{'}}=a^{i^{'}}(a^{l(n-1)}b)b^{-1}a^{-i^{'}}=a^{i^{'}+l(n-1)-i^{'}}$
$~~~~~~~~~~~~~~~~~~~~~=a^{l(n-1)}=a^{ln}.a^{-l}=(a^{2n})^{\frac{l}{2}}.a^{-l}=(a^{m})^{\frac{l}{2}}.a^{-l}=a^{-l}$
so, $a^{j^{'}}b$ takes $a^{l}~~ \mbox{to}~~a^{-l}$ $\Rightarrow \{a^{j},~a^{-j}\}$ is an orbit $\Leftrightarrow a^{j^{'}}b \in H$ which is true $\Leftrightarrow~~~l$  is even.

Now, for every element of $Q$ is in $G$-orbit of $1 \in Q.$
Since, every element of $Q$ are of the form $a^{l}\in Q.$ and $G$ has two type of elements $a^{i}$ and $a^{i}b$.
Now let us supposed that $~~~q_{1}=a^{l}=a^{i}.\theta(a^{i})^{-1}~~~~~~~~\mbox{for some}~a^{i}\in G$, thus
$~~~~~~~~~~~~~~~~q_{1}=a^{i}.\theta(a^{-i})=a^{i}.a^{-ri+[2-n]y}=a^{i-ri+[2-n]y}=a^{-i(r-1)+(2-n)y}=a^{v(r-1)+(2-n)y},~~~~~~~~~\mbox{where}~v=-i.$

Also,
$~~~~~~~~~~~~~~a^{l}=(a^{i}b).\theta(a^{i}b)^{-1}~~~~~~~~\mbox{for some}~a^{i}b\in G\Rightarrow$
            $~~~~~~~~~~~~~~a^{l}=(a^{i}b).\theta(a^{i}b)=(a^{i}b).(a^{ri+y}b)=a^{i}(ba^{ri+y})b=a^{i}.a^{(ri+y)(n-1)}b.b=a^{i+(ri+y)(n-1)}b^{2}$
Let us denote this element $~~~q_{2}=a^{l}=a^{-v+(y-vr)(n-1)},~~~~~~~~~~~\mbox{where}~~v=-i.$

            For $a^{-v}\in G,$ we have
            $a^{-v}.1.\theta(a^{-v})^{-1}=a^{-v}.\theta(a^{v})=a^{-v}.a^{rv+(2-n)y}=a^{v(r-1)+(2-n)y}=q_{1}.$
            Now for $~a^{-v}b \in G,$ we have
            $(a^{-v}b).1.\theta(a^{-v}b)^{-1}=(a^{-v}b).\theta(a^{-v}b)=(a^{-v}b).a^{-rv+y}b=a^{-v}(ba^{-rv+y})b$
            $~~~~~~~~~~~~~~~~~~~~~~~~=a^{-v}.a^{(-rv+y)(n-1)}b.b=a^{-v+(y-rv)(n-1)}=q_{2}.$

Therefore, for any $q\in Q$ there exist $g \in G$ such that $g.1.\theta(g)^{-1}=q.$
Hence $~~~~~~G \backslash Q=\{Q\}$, there is a single $G$-orbit on $Q.$
\end{proof}
\begin{exple}
Revisiting $\theta_{1},~\theta_{2}$ of Example \ref{f21} applying Proposition \ref{f20} to obtain that for $\theta_{1}$. Since all $l$ are even, therefore
$H_{1}\backslash Q_{1}=\{\{1\},~\{a^{4},a^{-4}\},\{a^{8},a^{-8}\}\}$,
            for $\theta_{2}$ there exists even and odd powers of elements of $Q$.
            For all even powers
            $H_{2}\backslash Q_{2}=\{\{1\},~\{a^{2},a^{-2}\},\ldots , \{a^{10},a^{-10}\},~\{a,a^{-1}\}, ~\{a^{5},a^{-5}\},~\{a^{9},a^{-9}\}\}$
            and odd powers
            $H_{2}\backslash Q_{2}=\{\{1\},~\{a^{2}\},\ldots, \{a^{10}\},~\{a\},~\{a^{5}\},~\{a^{9}\}\}.$
\end{exple}


\section{Involutions in  $Aut(QD_{2m})$}
If $\theta=rx+s$ is an automorphism such that $\theta^{2}=$id, then Proposition \ref{f2} gives $r\in R^{2}_{m}$ and $s\in ZD\mbox{iv}(r+1).$

\begin{prop}

Suppose that $m=4p^{t_{1}}_{1}p^{t_{2}}_{2},$ where the $p_{i}$ are distinct odd primes. Then $\mid R^{2}_{m} \mid =8$
\end{prop}
 \begin{proof}
  Since, $m=4p^{t_{1}}_{1}p^{t_{2}}_{2}~ \Rightarrow n=2p^{t_{1}}_{1}p^{t_{2}}_{2}.$ Suppose that $r\in \mathbb{Z}_n$ with $r^{2}\equiv1~(\mbox{mod~n})$ $~~~ \because \gcd(r,m)=1 \Rightarrow \gcd(r,4p^{t_{1}}_{1}p^{t_{2}}_{2})=1.$ But $m$ is even therefore, $r$ is odd. Assume that $r=2q+1$ for some $0\leq q < (n-1).$ Now, $r^{2}=(2q+1)^{2}=4q^{2}+4q+1$ $\Rightarrow 4q^{2}+4q\equiv 0~~(\mbox{mod~n})$ $\Rightarrow 4q(q+1)=l.(2p^{t_{1}}_{1}p^{t_{2}}_{2})$ $\because \gcd(4, 2p^{t_{1}}_{1}p^{t_{2}}_{2})=2$ $ \Rightarrow 2q(q+1)= l.(p^{t_{1}}_{1}p^{t_{2}}_{2})$ $(\mbox{mod}~p^{t_{1}}_{1}p^{t_{2}}_{2}).$ Either $q$ or $q+1$ is odd.
  \begin{enumerate}
  \item If $q$ is even, then $q=\frac{l}{2(q+1)}.p^{t_{1}}_{1}p^{t_{2}}_{2}$ $\Rightarrow q=h.p^{t_{1}}_{1}p^{t_{2}}_{2}$ where $1\leq h\leq4.$
Now, $r=2q+1~~\Rightarrow r=2h.p^{t_{1}}_{1}p^{t_{2}}_{2}+1$ $~~\Rightarrow r\equiv 1 (\mbox{mod~n})$
\item If $q+1$ is even, then $q+1=\frac{l}{2q}.p^{t_{1}}_{1}.p^{t_{2}}_{2}$ $\Rightarrow q+1=h.p^{t_{1}}_{1}p^{t_{2}}_{2}$ $\Rightarrow q=h.p^{t_{1}}_{1}p^{t_{2}}_{2}-1$ where $1\leq h\leq4.$
Now, $r=2q+1~~\Rightarrow r=2h.p^{t_{1}}_{1}p^{t_{2}}_{2}-1$ $~~\Rightarrow r\equiv -1 (\mbox{mod~n})$ $\Rightarrow r\equiv \pm 1 (\mbox{mod~n})$,
\end{enumerate}
\end{proof}

\begin{prop}

Suppose that $m=2^{\alpha}p^{t_{1}}_{1}p^{t_{2}}_{2} \ldots p^{t_{q}}_{q}$ where the $p_{i}$ are distinct odd primes. Then

$$\mid R^{2}_{m} \mid =\left\{
                         \begin{array}{ll}
                           2^{q+1} & \hbox{$\alpha= 2$} \\
                           2^{q+2} & \hbox{$\alpha= 3$} \\
                           2^{q+3} & \hbox{$\alpha \geq 4$.}
                         \end{array}
                       \right.
 $$
\end{prop}

 \begin{proof} Similar to the Propositions 10.1.
\end{proof}


\begin{expl}
Let $G=QD_{192}$, then $U_{96}=\{ 1,5,7,11,13,17,19,23,25,29,31,35,37,$ $41,43,47,49, 53,55,59,61,65,67,71,73,77,79,83,85,89,91,95\}$ and
$R_{96}^{2}=\{ 1,7,17,$ $23,25,31,41,47,49,55,65,71,73,79,89,95\}$ thus, $\mid R_{96}^{2}\mid=2^{4}=16$


\end{expl}
\begin{cor}
$$|Aut_{2}(QD_{2m})|=\sum_{r\in R^{2}_{m}}\gcd(r+1,n).$$
\end{cor}

\begin{thm}
Let $r\in R^{2}_{m},$ then the followings hold:
\begin{enumerate}
\item $\langle r-1 \rangle \leq ZD\mbox{iv}(r+1)$
\item $N_{r}=\mid \frac{ZD\mbox{iv}(r+1)}{\langle r-1 \rangle} \mid \leq 2$
\item $N_{r}=2$ if and only if $m=2^{\alpha}l~;~ \alpha \geq 2,$ and $l$ is odd, $ r\equiv \pm 1~(mod~2^{\alpha-1}).$
 \end{enumerate}
\end{thm}
\begin{proof}
\begin{enumerate}
  \item Given that $r\in R^{2}_{m},$ then $(r-1).(r+1)=~r^{2}-1\equiv 0$ $\Rightarrow (r-1) \leq ZD\mbox{iv}(r+1)$ $~~~\Rightarrow \langle r-1 \rangle \leq ZD\mbox{iv}(r+1).$
  \item Suppose
  \begin{equation}\label{e4}
    \mid \frac{ZD\mbox{iv}(r+1)}{\langle r-1 \rangle} \mid = j
\end{equation}
 It is well known that $\mid \langle r-1 \rangle \mid =\frac{n}{\gcd(r-1,n)}$ and $\mid ZD\mbox{iv}(r+1) \mid =\gcd(r+1,n)$

put values in \ref{e4} $$\frac{\gcd(r+1,n)}{\frac{n}{\gcd(r-1,n)}}=j$$
\begin{equation}\label{e5}
    \Rightarrow~~~\gcd(r+1,n).\gcd(r-1,n)=n.j
\end{equation}
Now, suppose that $m=2^{\alpha}p^{t_{1}}_{1}p^{t_{2}}_{2} \ldots p^{t_{q}}_{q}$ where $\alpha \geq 2~\Rightarrow \alpha -1 \geq 1$

$~~~~~~~~~~~~~~~\Rightarrow n=2^{\alpha -1}p^{t_{1}}_{1}p^{t_{2}}_{2} \ldots p^{t_{q}}_{q}$ where $p_{i}$ are distinct odd primes and $t_{i}\geq 1.$ Then  $$~\gcd(r+1,n)=2^{\alpha_{1}}p^{a_{1}}_{1}p^{a_{2}}_{2} \ldots p^{a_{q}}_{q}$$ with $\alpha_{1}\leq (\alpha-1)~$ and $~0\leq a_{i}\leq t_{i}~ \forall~i$ and $$\gcd(r-1,n)=2^{\alpha_{2}}p^{b_{1}}_{1}p^{b_{2}}_{2} \ldots p^{b_{q}}_{q}$$ with $\alpha_{2}\leq (\alpha-1)~$ and $~0\leq b_{i}\leq t_{i}. $ Now, for all $i \in \{1,2,\ldots q \}$ either $a_{i}=0~$ or $~b_{i}=0.$ Indeed, if $a_{i}>0~$ and $~b_{i}>0$ then $p_{i}$ divides both $(r+1)~$ and $~(r-1)$ which is impossible. $~~~\because (p_{i}>2)$

Similarly, either $\alpha_{1}\leq1~~or~~\alpha_{2}\leq1.$ Otherwise $2^{min\{\alpha_{1},\alpha_{2}\}}$ divides $(r+1)~$ and $~(r-1)$ which is impossible.

Since, from \ref{e5} $$\gcd(r+1,n).\gcd(r-1,n)=n.j$$ $$(2^{\alpha_{1}}p^{a_{1}}_{1}p^{a_{2}}_{2} \ldots p^{a_{q}}_{q}).(2^{\alpha_{2}}p^{b_{1}}_{1}p^{b_{2}}_{2} \ldots p^{b_{q}}_{q})=j.2^{\alpha -1}p^{t_{1}}_{1}p^{t_{2}}_{2} \ldots p^{t_{q}}_{q}$$ $$(2^{\alpha_{1}+\alpha_{2}}p^{a_{1}+b_{1}}_{1}p^{a_{2}+b_{2}}_{2} \ldots p^{a_{q}+b_{q}}_{q})=j.2^{\alpha -1}p^{t_{1}}_{1}p^{t_{2}}_{2} \ldots p^{t_{q}}_{q}$$ Now, since $\forall~i$ either $a_{i}=0~$ or $~b_{i}=0$ $ \Rightarrow a_{i}+b_{i}=t_{i}$ $$2^{\alpha_{1}+\alpha_{2}}=j.2^{\alpha-1}\frac{p^{t_{1}}_{1}p^{t_{2}}_{2} \ldots p^{t_{q}}_{q}}{p^{a_{1}+b_{1}}_{1}p^{a_{2}+b_{2}}_{2} \ldots p^{a_{q}+b_{q}}_{q}}$$
\begin{equation}\label{e6}
    2^{\alpha_{1}+\alpha_{2}}=j.2^{\alpha-1}.
\end{equation}
 As $\alpha_{1}\leq (\alpha-1)$ and $\alpha_{2}\leq (\alpha-1)$ and since either $\alpha_{1}\leq 1$ or $\alpha_{2}\leq 1,$ $ \Rightarrow (\alpha-1)\leq (\alpha_{1}+\alpha_{2})\leq \alpha$

equation \ref{e6} $\Rightarrow ~~~j=2^{\alpha_{1}+\alpha_{2}-\alpha+1}$ $ ~~~\Rightarrow1\leq j \leq2.$
  \item By Theorem \ref{f16}, $N_{r}$ is the number of divisors of $$\frac{\gcd(r-1,n).\gcd(r+1,n)}{n}$$ and $N_{r}$ is the number of divisor of $j$ computed above. Therefore,

      $N_{r}=1~~$ if $~j=1~~$ and $~~N_{r}=2~~$ if $~~j=2.$

      In either case, $$\mid \frac{ZD\mbox{iv}(r+1)}{\langle r-1 \rangle} \mid \leq 2.$$ Finally according to the proof of $(2),$ $j=2$ if and only if

      $~\alpha_{1}=\alpha-1,~\alpha_{2}=1~~~$ or $~~\alpha_{1}=1,~\alpha_{2}=\alpha-1.$

      If $~\alpha_{1}=\alpha-1,~\alpha_{2}=1$ then $2^{\alpha-1}$ divides $(r+1).$

      If $~\alpha_{1}=1,~\alpha_{2}=\alpha-1$ then $2^{\alpha-1}$ divides $(r-1).$

      $~\Rightarrow~~ r\equiv \pm 1~(\mbox{mod}~2^{\alpha-1}).$
\end{enumerate}
\end{proof}
\begin{expl}
Let $G=QD_{48}$ where $m=2^{3}.3$, $\alpha=3$, then \\ $U_{24}=\{ 1,5,7,11,13,17,19,23\}$ and $ R_{24}^{2}=\{ 1,5,7,11,13,17,19,23\}.$
\begin{enumerate}
\item Let $r=5 \in R_{24}^{2}$ then $\langle r-1\rangle=\langle 4\rangle= \{0,4,8\}=B_{1}$ and

$ZDiv(r+1)=ZDiv6=\{0,2,4,6,8,10\}=B_{2}$ clearly, $\langle 4\rangle\leq ZDiv6$.
\item For $r=5,$ $ \frac{ZD\mbox{iv}(6)}{\langle 4\rangle}=\frac{\{0,2,4,6,8,10\}}{\{0,4,8\}}=\frac{B_{2}}{B_{1}}=x+B_{1}~~~\forall~~ x \in B_{2}$

$0+B_{1}=\{0,4,8\}=B_{1}$, $2+B_{1}=\{2,6,10\}=B_{3}$

$4+B_{1}=\{4,8,0\}=B_{1},$ $6+B_{1}=\{6,10,2\}=B_{3}$

$8+B_{1}=\{8,0,4\}=B_{1},$ $10+B_{1}=\{10,2,6\}=B_{3}$

$U_{24}$ action on $\frac{ZDiv6}{\langle4 \rangle}=\{B_{1},B_{2}\}$
$\Rightarrow N_{r}=N_{5}=2$.
\item Since, $m=2^{3}.3$ $\Rightarrow l=3>1 $ is odd and $r=5\equiv1~~(mod~2^{2})$. \\ Thus $N_{5}=2$ as calculated in (2).
\end{enumerate}
\end{expl}
\begin{cor}
Suppose that $m=2^{\alpha}p^{t_{1}}_{1}p^{t_{2}}_{2} \ldots p^{t_{q}}_{q},$ where the $p_{i}$ are distinct odd primes and $\alpha\geq 2$. Then the number of equivalence classes $C_{4k}$ of $Aut_{2}(QD_{8k})$ is given by
$$C_{4k} = \left\{
           \begin{array}{ll}
             2^{q+1} & \hbox{$\alpha=2$} \\
             2^{q+2} & \hbox{$\alpha\geq 3$.}
           \end{array}
         \right.$$

\end{cor}
\begin{expl}
Let $G=QD_{384}$. 
Thus $ R_{192}^{2}=\{ 1,17,31,47,49,65,79,95,97,113,127,$ $143,145,161,175,191\} \Rightarrow \mid R_{192}^{2}\mid=2^{4}$.
It is easy to check that the values $r=1,65,97,161$ satisfy $r\equiv1~(mod~2^{5})$ and the values $r=31,95,127,191$ satisfy $r\equiv-1~(mod~2^{5})$.
So, $C_{m}=\mid\{1,31,65,95,97,127,161,191\}\mid=2^{3}$.
\end{expl}
Now for $\theta$ an involution, the set of twisted involutions is given in the followings porosities:
\begin{prop}
If $\theta^{2}=$id, then the set of twisted involutions is given as\\
$R=\{a^{i}\mid (r+1)i\equiv -2s~(mod~n)\} \cup \{a^{i}b \mid (r-1)i\equiv -s~(mod~n)\}$

\end{prop}
 \begin{proof}

Since, $\theta$ is an involution therefore, $r\in R^{2}_{m}~\mbox{and}~~s\in ZD\mbox{iv}(r+1).$
To find the set $R,~~~$, let $~a^{i}\in G,$ then $~~\theta(a^{i})=(a^{i})^{-1}$ $\Rightarrow a^{ri+(2-n)s}=a^{-i}$
$\Rightarrow ri+(2-n)s=-i$ $\Rightarrow(r+1)i+(2-n)s=0$ $\Rightarrow (r+1)i\equiv -2s~(\mbox{mod}~n).$

Now for $a^{i}b\in G,$ we have,
$\theta(a^{i}b)=(a^{i}b)^{-1}$ $~~~~\Rightarrow a^{ri+s}b=a^{i}b$ $~~~~\Rightarrow ri+s=i$
$\Rightarrow(r-1)i\equiv -s~~~(\mbox{mod}~n)$, hence
$R=\{a^{i}\mid (r+1)i\equiv -2s~(\mbox{mod}~n)\} \cup \{a^{i}b \mid (r-1)i\equiv -s~(\mbox{mod}~n)\}.$
\end{proof}
\begin{cor}
If $\theta^{2}=$id, $m=4k$, $n=2k$ and $k$ is even number, then $Q$ is a subgroup of $\langle a\rangle$ and the natural isomorphism $\psi : \langle a\rangle\longrightarrow \mathbb{Z}_{n}$ gives using Theorem \ref{f4} and Corollary \ref{f7} and \ref{f9}
\begin{enumerate}
  \item If $\mid Q\mid=2$, then
$\psi(Q)=\langle k\rangle.$

\item If $\mid Q \mid=\frac{k}{2}$  then
$$\psi(Q)=\left\{
\begin{array}{ll}
\langle d_{1}\rangle & \hbox{If \ $n\nmid r+2s-1$ and $d_{1}=\gcd(r+2s-1,n)$ } \\
\langle 4\rangle & \hbox{If \ $n\mid r+2s-1$ and $n=\gcd(r+2s-1,n)$.}
\end{array}
\right.
$$
\item If $\mid Q\mid=k$ and $\gcd(r+2s-1,n)=d_{1}$ then

$$\psi(Q)=\left\{
\begin{array}{ll}
\langle \frac{d_{1}}{2}\rangle & \hbox{If \{$(r+2s-1)\cap [l(1-r)-s]\}=\emptyset$} \\
\langle d_{1}\rangle & \hbox{If \{$(r+2s-1)\cap [l(1-r)-s]\}\not=\emptyset$,}
\end{array}
\right.
$$
where $l$ is an even number.
\end{enumerate}
\end{cor}
\begin{cor}
If $\theta^{2}=$id, $m=4k$, $n=2k$ and $k$ is even number, then $Q$ is a subgroup of $\langle a\rangle$ and the natural isomorphism $\psi : \langle a\rangle\longrightarrow \mathbb{Z}_{n}$ gives using Theorem \ref{f4} and Corollary \ref{f8} and \ref{f10}
\begin{enumerate}
  \item If $s$ is even, then $$\psi(Q)=\left\{
   \begin{array}{ll}
   \langle 0\rangle & \hbox{If $r=1,~s=0,~\mid Q\mid=1$} \\
   \langle 2\rangle & \hbox{If $\mid Q\mid=k$.}
   \end{array}
   \right.
$$
  \item If $s$ is odd, then $$\psi(Q)=\left\{
  \begin{array}{ll}
  \langle k\rangle & \hbox{$\mid Q\mid=2~ \mbox{If}~ n\mid 1-r ~\mbox{and}~ n\mid s-k $} \\
  \langle 1\rangle & \hbox{$\mid Q\mid=n~\mbox{If} \ n\nmid r+2s-1 \ \mbox{and} \ n\nmid 1-r$.}
  \end{array}
  \right.
$$
\end{enumerate}
\end{cor}
\begin{cor} In any Case $R\neq Q.$
\end{cor}

\end{document}